\documentclass[preprint,12pt]{elsarticle}

\usepackage{lipsum} % For generating dummy text, remove in the final version
\usepackage{hyperref} % For hyperlinks
\usepackage{geometry} % For setting page margins
\usepackage{graphicx}
\usepackage{multirow}
\usepackage{amsmath,amssymb,amsfonts}
\usepackage{amsthm}
\usepackage{mathrsfs}
\usepackage[title]{appendix}
\usepackage{xcolor}
\usepackage{textcomp}
\usepackage{manyfoot}
\usepackage{booktabs}
\usepackage{algorithm}
\usepackage{algorithmicx}
\usepackage{algpseudocode}
\usepackage{listings}
\usepackage{cleveref}
\usepackage{hyperref}
\theoremstyle{plain}
\newtheorem{theorem}{Theorem}[section]
\newtheorem{proposition}[theorem]{Proposition}

\newtheorem{lemma}[theorem]{Lemma}

\theoremstyle{remark} % Or use plain/definition for different styles
\newtheorem{note}{Note}[section]
\theoremstyle{definition}

\newtheorem{remark}{Remark}[section]
\newtheorem{definition}{Definition}[section]
\begin{document}

\begin{frontmatter}



\title{Exact Null Controllability of Non-Autonomous Conformable Fractional Semi-Linear Systems with Nonlocal Conditions}


\author[inst1]{Dev Prakash Jha}
\ead{devprakash.22@res.iist.ac.in}
\affiliation[inst1]{organization={Department of Mathematics, Indian Institute of Space Science and Technology},%Department and Organization
            addressline={Valiamala P.O.}, 
            city={Thiruvananthapuram},
            postcode={695547}, 
            state={Kerala},
            country={India}}
\author[inst1]{Raju K. George}
\ead{george@iist.ac.in}
\begin{abstract}
We study the exact null controllability of a class of non-autonomous conformable fractional semi-linear evolution systems with nonlocal initial conditions in Hilbert spaces. The analysis is carried out within the framework of conformable fractional calculus and linear evolution operator theory. Under suitable assumptions, we establish the existence of mild solutions and provide sufficient conditions for exact null controllability. Notably, the nonlocal term is allowed to be continuous without requiring compactness or Lipschitz-type conditions. An example is included to illustrate the applicability of the main results.
\end{abstract}


%\item We propose a method to initialize the initial conditions of an abstract system.
%\item Establish the existence and uniqueness of evolution operators for non-autonomous systems.
%\item Develop the existence and uniqueness of mild solutions for semi-linear conformable fractional-order systems.
%\item Obtain exact null controllability for semi-linear conformable fractional systems.
%\item  We provide a comprehensive example to demonstrate the applicability of the established  theoretical results.

%\item Establish existence and uniqueness of evolution
%operators for non-autonomous conformable fractional-order %homogeneous systems
%\item Establish mild solutions to semi-linear conformable %fractional order systems
%using linear evolution operators
%\item Obtain exact null controllability of semi-linear conformable %fractional order systems
%\item We demonstrate the results with an example


\end{frontmatter}



\section{Introduction}
In this study, we explore the controllability of a specific class of non-autonomous fractional parabolic differential equations subject to nonlocal conditions within Hilbert spaces $Z$. Our investigation centers on systems modeled by the following nonlinear fractional differential equation:

\begin{equation}
\left.
\begin{array}{l}
\left( T_{\alpha} x \right)(t) + A(t)x(t) = Bu(t) + F(t,x(b(t)) ), \quad \alpha \in (0,1] \\
x(\zeta)+g(x) = x_{0}, \quad t_1  \leq t \leq t_{2} < \infty,\quad \zeta \in [t_0,t_f] \cap [t_1,t_2], 
\end{array}
\right\}
\tag{P}
\label{eq:P}
\end{equation}

where,

\begin{itemize}
    \item The function \( x(t) \) denotes the system's state in the space \( Z \);
    \item The control input \( u(t) \) lies in the space \( L^2([\zeta, t_2]; U) \), which consists of admissible control functions, where \( U \) is a real Hilbert space;
    \item The conformable fractional derivative of a function \( f \) of order \( \alpha \) at  \( t \), written as \( \left(T_{\alpha} f\right)(t) \), is defined by
    $$
    \left( T_{\alpha} f \right)(t) = \lim_{\epsilon \rightarrow 0} \frac{f\left(t + \epsilon(t)^{1 - \alpha}\right) - f(t)}{\epsilon}, \quad \text{for } t > 0;
    $$
    \item Let \( A(t): Z \rightarrow Z^* \) be a closed, possibly unbounded linear operator. Following the result by Jha~\cite{jha2024existence}, the operator family \( \{A(t)\}_{\frac{t^\alpha}{\alpha} \in [t_0, t_f]} \) generates an almost strongly continuous evolution operator \( \Psi_\alpha: \mathcal{D} \rightarrow \text{BL}(Z) \), where
    \[
    \mathcal{D} = \left\{ \left( \frac{t^\alpha}{\alpha}, \frac{s^\alpha}{\alpha} \right) \in [t_0, t_f]^2 : t_0 \leq \frac{s^\alpha}{\alpha} \leq \frac{t^\alpha}{\alpha} \leq t_f \right\},
    \]
    and \( \text{BL}(Z) \) represents the space of all bounded linear operators \( \Psi_\alpha(t, s) \) on \( Z \), satisfying \( \| \Psi_\alpha(t, s) \| \leq M \) whenever \( t > s \);
    \item The operator \( B \) is a bounded linear map from \( U \) to \( Z \);
    \item The function \( b(t) \) is an element of the space \( \mathbb{C}([\zeta, t_2]; [\zeta, t_2]) \), which consists of continuous functions over the interval.
\end{itemize}

\begin{note}
In the following section, we reformulate the system (\ref{eq:P}) in such a way that
\[
\frac{t^\alpha}{\alpha} \in [t_0, t_f] \iff t \in \left[(\alpha t_0)^{\frac{1}{\alpha}}, (\alpha t_f)^{\frac{1}{\alpha}}\right] \iff t \in [t_1, t_2],
\]
where \( t_1 = (\alpha t_0)^{1/\alpha} \) and \( t_2 = (\alpha t_f)^{1/\alpha} \).\\
This change is justified since the time variable \( t \) is taken to be a positive real number.
\end{note}

Byszewski~\cite{byszewski1991theorems} analyzed the nonlocal Cauchy problem and underscored the relevance of nonlocal conditions in numerous applications, as detailed in~\cite{byszewski1991theorems,deng1993exponential} and the references therein. For instance, in~\cite{deng1993exponential}, the diffusion process of a small quantity of gas inside a transparent tube is modeled using the following expression:

 \begin{equation}\label{eq:g(x)}
     g(x) = \sum_{i=1}^{q} c_i x(t'_i),
 \end{equation}
where \( c_i \), for \( i = 1, 2, \dots, q \), are constants and \( \zeta  < t'_1<t'_2 < \dots < t'_q < a \). In this framework, Eq.~(\ref{eq:g(x)}) incorporates additional data points measured at \( t'_i \), for \( i = 1,2, \dots, q \). Over the past several years, a significant number of researchers have investigated the existence, regularity, stability, and (approximate) controllability of abstract differential and functional differential Cauchy evolution problems subject to nonlocal conditions (see \cite{ezzinbi2007existence,fu2011approximate,guo2004controllability,liang2004nonlocal,boucherif2009semilinear,chen2013monotone,xiao2005existence,wang2017approximate,mahmudov2008approximate} ).

This study is devoted to addressing the problem of exact null controllability for the non-autonomous functional evolution system described by (\ref{eq:P}) under nonlocal constraints. We begin by establishing a boundedness property of the operator \( L_\xi^{-1} N_\xi^{t_2} \), drawing on techniques from Dauer \cite{dauer2004exact}. This foundational result facilitates our subsequent analysis of the system's exact null controllability, employing the framework of linear evolution operator theory. Notably, the nonlocal condition considered here is defined as a constant function within a narrow interval \( U(t = \zeta; \delta) \). As such, we do not impose the usual compactness or Lipschitz assumptions on the function \( g \) involved in the nonlocal term. Instead, we assume that \( g \) is continuous and explicitly defined over the interval \( [\delta, a] \), where \( \delta > 0 \) is small. Such conditions are commonly replaced by compactness or Lipschitz constraints in much of the existing literature---see, e.g., \cite{ezzinbi2007existence,mahmudov2008approximate,fu2006existence,guo2004controllability}.

This paper is structured as follows: Section~\ref{sec:Preliminaries} provides essential definitions and foundational results. In Section~\ref{sec:Exact Null Controllability}, we introduce key notations, assumptions, and preliminary findings concerning linear evolution operators and the exact null controllability of equation~(\ref{eq:P}). Lastly, Section~\ref{sec:Application} illustrates the practical relevance of the theoretical results through a concrete example.

\section{Preliminaries}

\label{sec:Preliminaries}
\begin{definition}(Conformable Fractional Derivative \cite{bib1})  
Let \(f: [0, \infty) \rightarrow \mathbb{R}\) be a real-valued function. The conformable fractional derivative of order \(\alpha \in (0,1]\) is defined by  
\begin{equation}  
T_{\alpha}(f)(t) = \lim_{\epsilon \to 0} \frac{f(t + \epsilon t^{1 - \alpha}) - f(t)}{\epsilon}, \quad t > 0.  
\label{eq:P2}  
\end{equation}  

If this limit exists, the function \(f\) is called \(\alpha\)-differentiable at \(t\), and the derivative is denoted as \(f^{(\alpha)}(t) = T_{\alpha}(f)(t)\). Furthermore, if \(f\) is \(\alpha\)-differentiable on an interval \((0, a)\) for some \(a > 0\), and the limit \(\lim_{t \to 0^{+}} f^{(\alpha)}(t)\) exists, then we define \(f^{(\alpha)}(0) = \lim_{t \to 0^{+}} f^{(\alpha)}(t)\).
\end{definition}

\begin{definition}(Left Conformable Fractional Derivative \cite{bib2})  
Suppose \(f: [a, \infty) \rightarrow \mathbb{R}\) and \(\alpha \in (0,1]\). The left conformable fractional derivative of \(f\) starting from the point \(a\) is defined as  
\begin{equation}  
T_{\alpha}^{a}(f)(t) = \lim_{\epsilon \to 0} \frac{f(t + \epsilon (t - a)^{1 - \alpha}) - f(t)}{\epsilon}, \quad t > a.  
\label{eq:P3}  
\end{equation}  

This derivative is commonly written as \(f^{(\alpha, a)}(t)\). In the case where \(a = 0\), we write \(T_{\alpha}^{a} f\) as \(T_{\alpha} f\). If \(f^{(\alpha, a)}(t)\) exists on an interval \((a, b)\), then the derivative at the starting point is defined via the limit: \(f^{(\alpha, a)}(a) = \lim_{t \to a^{+}} f^{(\alpha, a)}(t)\).
\end{definition}

\begin{definition}(Right Conformable Fractional Derivative \cite{bib2})  
Let $\alpha \in (0,1]$ and suppose $f$ is defined on an interval ending at $b$. The \textit{right conformable fractional derivative} of order $\alpha$ at a point $t < b$ is defined by  
\begin{equation}
\left({ }_{\alpha}^{b} \mathrm{T} f\right)(t) = -\lim_{\epsilon \to 0} \frac{f\left(t + \epsilon(b - t)^{1 - \alpha}\right) - f(t)}{\epsilon}, \quad t < b.
\label{eq:P4}
\end{equation}  
If this limit exists for every $t \in (a, b)$ where $b > a$, then the derivative at $t = b$ is defined as  
\[
\left({ }_{\alpha}^{b} \mathrm{T} f\right)(b) = \lim_{t \to b^{-}} \left({ }_{\alpha}^{b} \mathrm{T} f\right)(t).
\]
\end{definition}

\begin{theorem}(\cite{bib1})\label{thm:th1}  
Assume $\alpha \in (0,1]$ and that the functions $f_1$ and $f_2$ are $\alpha$-differentiable at a point $t > a$. Then the following rules hold:
\begin{enumerate}
    \item \textbf{Linearity:} For any constants $c, d \in \mathbb{R}$,
    \[
    T_{\alpha}^{a}(c f_1 + d f_2) = c T_{\alpha}^{a}(f_1) + d T_{\alpha}^{a}(f_2).
    \]
    \item \textbf{Constant function:} For any constant $\beta$,
    \[
    T_{\alpha}^{a}(f) = 0, \quad \text{where } f(t) = \beta.
    \]
    \item \textbf{Product rule:}
    \[
    T_{\alpha}^{a}(f_1 f_2) = f_1 T_{\alpha}^{a}(f_2) + f_2 T_{\alpha}^{a}(f_1).
    \]
    \item \textbf{Quotient rule:}
    \[
    T_{\alpha}^{a}\left(\frac{f_1}{f_2}\right) = \frac{f_2 T_{\alpha}^{a}(f_1) - f_1 T_{\alpha}^{a}(f_2)}{f_2^2}.
    \]
    \item \textbf{Relation to classical derivative:} If $f_1$ is differentiable in the classical sense, then
    \[
    T_{\alpha}^{a}(f_1)(t) = (t - a)^{1 - \alpha} \frac{d f_1}{d t}(t).
    \]
    \item \textbf{Example:} For $f(t) = \frac{(t - a)^{\alpha}}{\alpha}$,
    \[
    T_{\alpha}^{a}(f)(t) = 1.
    \]
\end{enumerate}
\end{theorem}

\begin{definition}[Chain Rule \cite{bib2}]  
Let $f$ and $g$ be two functions defined on the interval $(a, \infty)$, both of which are (left) $\alpha$-differentiable for $0 < \alpha \leq 1$. Define $h(t) = f(g(t))$. Then $h(t)$ is also (left) $\alpha$-differentiable. For $t \neq a$ and $g(t) \neq 0$, we have  
\begin{equation}
\left(T_{\alpha}^{a} h\right)(t) = \left(T_{\alpha}^{a} f\right)(g(t)) \left(T_{\alpha}^{a} g\right)(t)(g(t) - a)^{\alpha - 1},  
\label{eq:P5}  
\end{equation}  
and at $t = a$, the derivative is defined as  
\begin{equation}
\left(T_{\alpha}^{a} h\right)(a) = \lim _{t \rightarrow a^{+}} \left(T_{\alpha}^{a} f\right)(g(t)) \left(T_{\alpha}^{a} g\right)(t) (g(t) - a)^{\alpha - 1}.  
\label{eq:P6}  
\end{equation}  
\end{definition}

\begin{definition}[$\alpha$-Conformable Integral \cite{bib1}]
The $\alpha$-conformable integral of a function $f$ over $[a, t]$ is defined by
\[
I_\alpha^a (f)(t)=\int_a^t f(x) \, d(x,\alpha) = I_a^1 (t^{\alpha-1} f(x)) = \int_a^t \frac{f(x)}{(x-a)^{1-\alpha}} \, dx,
\]
where the integral is interpreted in the standard Riemann improper sense, and $\alpha \in (0, 1)$.
\end{definition}

\begin{theorem}[Schauder's Fixed Point Theorem]
Let $K$ be a closed, convex, and nonempty subset of a Banach space $E$. If $F : K \to K$ is a continuous mapping and $F(K)$ is relatively compact in $E$, then $F$ admits at least one fixed point in $K$.
\end{theorem}

\begin{theorem}[\cite{bib4}]\label{prop:pro4}
Suppose $U(t)$ is a regular and $\alpha$-differentiable matrix-valued function. Then the $\alpha$-derivative of its inverse is given by
\[
(U^{-1})^{(\alpha)}(t) = -U^{-1}(t) U^{(\alpha)}(t) U^{-1}(t).
\]
\end{theorem}

\begin{theorem}[Differentiation Under the Integral Sign {\cite{bib5}}]\label{thm:th5}
Suppose \( h(t, s) \) is a sufficiently smooth function. Then, the conformable fractional derivative of an integral with variable limits is given by:
\begin{align*}
   T_{\alpha}\left(\int_{a(t)}^{b(t)} h(t, s) \, d(s,\alpha)\right)(t)
   &=
   \int_{a(t)}^{b(t)} T_{\alpha}\left(h(t, s)\right)(t) \, d(s,\alpha) \\
   &\quad + h(t, b(t)) \cdot T_{\alpha}(b)(t) - h(t, a(t)) \cdot T_{\alpha}(a)(t).
\end{align*}
\end{theorem}

\section{Existence of Exact Null Controllability}
\label{sec:Exact Null Controllability}

%In this section we divide two subsections; one is \ref{subsec:Existence of the %Evolution Operator} and \ref{subsec:Existence of mild solutions}.

\subsection{Evolution Operator}
\label{subsec:Existence of the Evolution Operator}
Let \( Z \) denote a Hilbert space with norm \( \|\cdot\| \). We consider a collection of linear operators \( \{A(t) : t_0 \leq \frac{t^\alpha}{\alpha} \leq t_f\} \) that satisfy the following properties:

\begin{itemize}
    \item[(a)] The domain \( D_A \), which is common to all \( A(t) \), is dense in \( Z \), independent of \( t \), and each \( A(t) \) is a closed operator.

    \item[(b)] For each \( \frac{t^\alpha}{\alpha} \), the resolvent \( R_\alpha(\lambda_\alpha, A(t)) \) exists whenever \( \text{Re}(\lambda_\alpha) \leq 0 \), and there exists a constant \( K_1 > 0 \) such that
    \[
    \|R_\alpha(\lambda_\alpha, A(t))\| \leq \frac{K_1}{|\lambda_\alpha| + 1}.
    \]

    \item[(c)] There exist constants \( 0 < \delta \leq 1 \) and \( K_2 > 0 \) such that for all \( \frac{t^\alpha}{\alpha}, \frac{s^\alpha}{\alpha}, \frac{\tau^\alpha}{\alpha} \in [t_0, t_f] \), the inequality
    \[
    \|(A(t) - A(s))A^{-1}(\tau)\| \leq K_2 |t^\alpha - s^\alpha|^\delta
    \]
    holds.

    \item[(d)] For any \( \frac{t^\alpha}{\alpha} \in [t_0, t_f] \) and \( \lambda_\alpha \in \rho(A(t)) \), where \( \rho(A(t)) \) denotes the resolvent set of \( A(t) \), the operator \( R_\alpha(\lambda_\alpha, A(t)) \) is compact.
\end{itemize}

Under the given assumptions, the operator family \( \{A(t)\} \) generates a fractional evolution operator \( \Psi_\alpha(t, \tau) \), as introduced in~\cite{jha2024existence}, and defined by
\begin{equation}\label{eq:1}
    \Psi_\alpha(t,s) = e^{-\left(\frac{t^\alpha}{\alpha} - \frac{s^\alpha}{\alpha}\right)A(t)} + \int_{\frac{s^\alpha}{\alpha}}^{\frac{t^\alpha}{\alpha}} e^{-\left(\frac{t^\alpha}{\alpha} - \frac{\tau^\alpha}{\alpha}\right)A(\tau)}  R(\tau,s) \, d(\tau,\alpha),
\end{equation}
where \( \exp\left(-\frac{\tau^\alpha}{\alpha} A(t)\right) \) denotes the analytic semigroup generated by the operator \( -A(t) \).

\begin{proposition}\cite{jha2024existence}
The fractional evolution system \( \{ \Psi_\alpha(t,s): t_0 \leq \frac{s^\alpha}{\alpha} \leq \frac{t^\alpha}{\alpha} \leq t_f \} \) satisfies the following properties:
\begin{itemize}
    \item[($A_1$)] \label{item:A} For each \( t_0 \leq \frac{s^\alpha}{\alpha} \leq \frac{t^\alpha}{\alpha} \leq t_f \), the operator \( \Psi_\alpha(t,s) \in BL(Z) \), the space of bounded linear operators on \( Z \), and for every \( x \in Z \), the mapping \( (t, s) \mapsto \Psi_\alpha(t,s)x \) is continuous.
    
    \item[($A_2$)] \label{item:B} The composition law holds:
    \[
        \Psi_\alpha(t,s)\Psi_\alpha(s,\tau) = \Psi_\alpha(t,\tau),
    \]
    for all \( t_0 \leq \frac{\tau^\alpha}{\alpha} \leq \frac{s^\alpha}{\alpha} \leq \frac{t^\alpha}{\alpha} \leq t_f \).
    
    \begin{note}
        This property typically fails for conformable fractional derivatives with initial conditions of the form \( x(\zeta) = x_0 \), whether left- or right-sided.
    \end{note}
    
    \item[($A_3$)] \label{item:C} The identity condition holds: \( \Psi_\alpha(t,t) = I \), where \( I \) denotes the identity operator on \( Z \).
\end{itemize}
\end{proposition}

\begin{lemma}
Let $\{ \Psi_\alpha(t,s) : t_0 \leq \frac{s^\alpha}{\alpha} \leq \frac{t^\alpha}{\alpha} \leq t_f \}$ be a family of linear evolution operators. Then, $\Psi_\alpha(t,s)$ is a compact operator for all $t^\alpha - s^\alpha > 0$.
\end{lemma}

\begin{proof}
Compactness of the evolution operator $\Psi_\alpha(t,s)$ follows from condition (d), as established in Proposition 2.1 of \cite{fitzgibbon1978semilinear}. Consequently, there exists a constant $M \geq 1$ such that
\begin{equation}\label{eq:3}
   \| \Psi_\alpha(t,s) \| \leq M, \quad \text{for all} \quad t_0 \leq \frac{s^\alpha}{\alpha} \leq \frac{t^\alpha}{\alpha} \leq t_f.
\end{equation}
\end{proof}

\begin{theorem}\label{Theorem:Estimation}
Let \( 0 < \alpha \leq 1 \). Suppose \( t, \tau \) satisfy \( t^\alpha - \tau^\alpha > h \). Then there exists a constant \( C > 0 \) such that
\[
\|\Psi_\alpha(t + h, \tau) - \Psi_\alpha(t, \tau)\| \leq \frac{C}{t^\alpha - \tau^\alpha} \left[\frac{(t + h t^{1 - \alpha})^\alpha - t^\alpha}{\alpha} \right].
\]
\end{theorem}

\begin{proof}
The evolution operator \( \Psi_\alpha(t, \tau) \) is known to be strongly \(\alpha\)-conformably differentiable with respect to \( t \), and it satisfies
\[
T_\alpha\big( \Psi_\alpha(t, s) \big)(t) = -A(t) \Psi_\alpha(t, s),
\]
as established by Jha~\cite{jha2024existence} and Pazy~\cite[p.~158]{pazy2012semigroups}.

Consequently, we estimate:
\begin{align}
    \left\| \Psi_\alpha(t + h t^{1 - \alpha}, \tau) - \Psi_\alpha(t, \tau) \right\|
    &= \left\| \int_t^{t + h t^{1 - \alpha}} T_\alpha \big( \Psi_\alpha(s, \tau) \big)(s) \, d(s, \alpha) \right\| \\
    &\leq \int_t^{t + h t^{1 - \alpha}} \left\| T_\alpha \big( \Psi_\alpha(s, \tau) \big)(s) \right\| \, d(s, \alpha) \\
    &\leq \int_t^{t + h t^{1 - \alpha}} \frac{C}{s^\alpha - \tau^\alpha} \, d(s, \alpha).
\end{align}

Since \( s \in [t, t + h t^{1 - \alpha}] \) and \( t^\alpha - \tau^\alpha > h \), it follows that \( s^\alpha - \tau^\alpha \geq t^\alpha - \tau^\alpha \). Therefore:
\begin{align*}
    \int_t^{t + h t^{1 - \alpha}} \frac{C}{s^\alpha - \tau^\alpha} \, d(s, \alpha)
    &\leq \frac{C}{t^\alpha - \tau^\alpha} \int_t^{t + h t^{1 - \alpha}} d(s, \alpha) \\
    &= \frac{C}{t^\alpha - \tau^\alpha} \left[ \frac{(t + h t^{1 - \alpha})^\alpha - t^\alpha}{\alpha} \right].
\end{align*}

Thus, the desired estimate follows:
\[
\|\Psi_\alpha(t + h, \tau) - \Psi_\alpha(t, \tau)\| \leq \frac{C}{t^\alpha - \tau^\alpha} \left[ \frac{(t + h t^{1 - \alpha})^\alpha - t^\alpha}{\alpha} \right].
\]
\end{proof}

Moreover, based on Theorem~\ref{Theorem:Estimation}, we establish the following result, which will serve as a key component in the analysis presented in the next section.

\begin{lemma}\label{lemma:1}
    The family of operators $\{\Psi_\alpha(t,s):\frac{t^\alpha}{\alpha} > \frac{s^\alpha}{\alpha} \geq t_0\}$ is continuous with respect to \( t \), uniformly in \( s \), under the uniform operator topology.
\end{lemma}

To explore the exact null controllability of equation~(\ref{eq:P}), we consider the associated linear system:

\begin{equation}
\begin{cases}
T_{\alpha}z(t) + A(t)z(t) = Bu(t) + h(t), & t \in [t_1,t_2],\quad \alpha \in (0,1],\\
z(\zeta) = z_0, \quad \zeta \in [t_0,t_f] \cap [t_1,t_2],
\end{cases}
\label{eq:R}
\tag{R}
\end{equation}

This system is linked to equation~(\ref{eq:P}) and is considered under the condition that there exists a constant \( 0 \leq \xi < \min\{\alpha, \delta\} \) such that \( h \in L^{\frac{1}{\xi}}([\zeta, t_2]; Z) \).\\

We define the operators $L^{t_2}_{\zeta}: L^2([\zeta,t_2];U) \rightarrow Z$ and $N^{t_2}_{\zeta}: Z \times L^{\frac{1}{\xi}}([\zeta,t_2]; Z) \rightarrow Z$ as follows:

\[
\begin{aligned}
L^{t_2}_{\zeta}(u) &= \int_{\zeta}^{t_2} \Psi_\alpha(t_2, s)Bu(s) \, d(s,\alpha), \quad \text{for all } u \in L^2([\zeta,t_2];U), \\
N^{t_2}_{\zeta}(z_0, h) &= \Psi_\alpha(t_2, \zeta)z_0 + \int_{\zeta}^{t_2} \Psi_\alpha(t_2, s)h(s) \, d(s,\alpha), \\
&\text{for all } (z_0, h) \in Z \times L^{\frac{1}{\xi}}([\zeta,t_2]; Z).
\end{aligned}
\]

The following concepts are fundamental for our analysis:

\begin{definition}\label{Definition:1}
    The system described by equation~\eqref{eq:R} is said to be \emph{exactly null-controllable} on the interval $[\zeta, t_2]$ if the following inclusion holds:
    \begin{equation}\label{eq:4}
        \text{Im } L^{t_2}_{\zeta} \supseteq \text{Im } N^{t_2}_{\zeta}.
    \end{equation}
\end{definition}

\begin{remark}\label{remark:1}
    As established in~\cite{curtain2012introduction}, the system~\eqref{eq:R} is exactly null-controllable if and only if there exists a constant $\gamma > 0$ such that:
    \begin{equation}\label{eq:5}
        \| (L^{t_2}_{\zeta})^* z \| \geq \gamma \| (N^{t_2}_{\zeta})^* z \|, \quad \text{for all } z \in Z.
    \end{equation}
\end{remark}

We now state a key lemma essential to our discussion:

\begin{lemma}\label{lemma:2}\cite{jha2024existence}
    Assume that the system~\eqref{eq:R} is exactly null-controllable. Then the linear operator
    \[
    H_\alpha := (L_{\zeta})^{-1}(N_{\zeta}^{t_2}): Z \times L^{\frac{1}{\xi}}([\zeta, t_2]; Z) \to L^2([\zeta, t_2]; U)
    \]
    is bounded. Furthermore, the control function defined by
    \begin{align}\label{eq:6}
        u(t) &:= - (L_{\zeta})^{-1}(N_{\zeta}^{t_2}(z, h)(t)) \notag \\
        &= -H_\alpha(z_0, h)(t) \\
        &= - (L_{\zeta})^{-1} \left( \Psi_\alpha(t_2, \zeta) z_0 + \int_{\zeta}^{t_2} \Psi_\alpha(t_2, s) h(s) \, d(s, \alpha) \right)(t), \notag
    \end{align}
    drives the linear system~\eqref{eq:R} from the initial condition $z_0$ to the origin. Here, $L_{\zeta}$ is the restriction of $L_{\zeta}^{t_2}$ to the orthogonal complement of its kernel, denoted by $[\text{ker } L_{\zeta}^{t_2}]^\perp$.
\end{lemma}

\subsection{Existence of mild solutions }

\label{subsec:Existence of mild solutions }

This subsection investigates the precise null controllability issue related to equation (\ref{eq:P}). We first introduce the concept of a mild solution and then define its exact null controllability.

\begin{definition}\label{Definition:2}
  A function \( x \in \mathbb{C}([\zeta,t_2]; Z) \) is considered a mild solution to problem (\ref{eq:P}) if it satisfies the following equation:
\begin{equation}\label{eq:7}
x(t) = \Psi_\alpha(t,\zeta)[x_{0}-g(x)]  + \int_{\zeta}^{t} \Psi_\alpha(t,s)[Bu(s)+F(s,x(b(s)))] \frac{ds}{s^{1-\alpha}},\quad t \in [\zeta,t_2],
\end{equation}
where \( u(\cdot) \in L^2([\zeta,t_2]; U) \) is the control function.
\end{definition}

\begin{definition}\label{Definition:3}
    A system (\ref{eq:P}) is said to be exactly null controllable if there exists a control \( u \in L^2([\zeta,t_2]; U) \) such that, under this control, \( x(t_2, u) = 0 \).
\end{definition}

We now impose the following conditions on the semi-linear system (\ref{eq:P}):

\begin{itemize}
    \item[$(H_1)$] Let $t \in [\zeta, t_2]$. The function $F(t, \cdot): Z \rightarrow Z$ remains continuous. For any given $x \in Z$, the function $F(\cdot, x): [\zeta, t_2] \rightarrow Z$ is strongly measurable. Furthermore, there exist constants $0 \leq \xi < \min \{\alpha, \delta\}$ and $\gamma > 0$. For any $r > 0$, there exists a function $h_r(t) \in L^{\frac{1}{\xi}}([\zeta, t_2]; \mathbb{R}^+)$ such that
\[
\sup_{x \in W_r} \| F(t, x(b(t))) \| \leq h_r(t),
\]
holds for almost every $t \in [\zeta, t_2]$. Moreover, it satisfies
\[
\liminf_{r \to +\infty} \frac{\| h_r(t) \|_{L^\frac{1}{\xi}}}{r} = \gamma < \infty,
\]
where $W_r = \{ x \in \mathbb{C}([\zeta, t_2]; Z) : \| x(\cdot) \| \leq r \}$.

     \item[($H_2$)]  The function \( g: \mathbb{C}([\zeta, t_2], Z) \to Z \) is continuous, and there exists a non-decreasing continuous function \( \varphi_r: \mathbb{R}^+ \to \mathbb{R}^+ \) and a constant \( L > 0 \) such that, for some \( r > 0 \) and all \( x \in W_r \),
    \[
    \|g(x)\| \leq \varphi_r, \quad \text{and} \quad \liminf_{r \to +\infty} \frac{\varphi_r}{r} = L < +\infty.
    \]

\item[$(H_3)$] The delay function $b(\cdot) : [\zeta,t_2]  \to [\zeta,t_2]$ is continuous.
    \item[$(H_4)$] The linear system (\ref{eq:R}) is exactly null controllable on $[\zeta,t_2]$ in $Z$.
\end{itemize}
Next, for convenience, let us introduce the following notation:\\
$$\psi_r=\sup \{h_r(\tau):\lVert \tau \rVert \leq r\},\quad \textit{then} \quad \liminf_{r \rightarrow +\infty} \frac{\| \psi_r \|_{L^\frac{1}{\xi}}}{r} = \gamma ,$$ 

\[
N =  \left\{ \bigg(  (t_2)^{\frac{\alpha-\xi}{1-\xi}}-(\zeta)^{\frac{\alpha-\xi}{1-\xi}} \bigg)^{1-\xi} \bigg(\frac{1-\xi}{\alpha-\xi}\bigg)^{1-\xi}:0\leq \xi <\min{ \{\alpha, \delta\}}\right\}.
\]

\begin{theorem}\label{Theorem:1}
Suppose the assumptions $(H_1)-(H_4)$ hold. Then the nonlocal fractional differential system given by (\ref{eq:P}) is exactly null controllable on the interval $[\zeta, t_2]$, provided that
\begin{equation}\label{eq:8}
   M^2 L + \lVert B \rVert \lVert H_\alpha \rVert N M (ML + \gamma) + N M \gamma < 1,
\end{equation}
and there exists a positive constant $r$ such that  
\begin{equation}
  \lVert H \rVert\left( \lVert x_0\rVert + M \varphi_r + \psi_r \sqrt{t_2 - \zeta} \right) < r.  
\end{equation}
\end{theorem}

\begin{proof}
We now proceed to decompose the proof of the theorem into two logical segments:\\
\textbf{Part 1:}
   Let \( x \in \mathbb{C}([\zeta,t_2]; Z) \). We define the control \( u \) as
\begin{align}\label{eq:11}
u(t) &= -H_\alpha(x_0 - g(x), F)(t) \notag \\
     &= -(L_\xi)^{-1} \Bigg[\Psi_\alpha(t_2, \zeta)[x_0 - g(x)] 
      + \int_{\zeta}^{t_2} \Psi_\alpha(t_2, s) F(s, x(b(s))) \frac{ds}{s^{1-\alpha}} \Bigg](t).
\end{align}
Clearly, the definition of \( u(t) \) is valid, as we have \( x_0 - g(x) \in Z \) and \( F(s, x(b(s))) \in L^{\frac{1}{\xi}}([\zeta,t_2]; Z) \). This control successfully drives the initial condition \( x_0 \) to the target zero state. Specifically, if \( x(t,u) \) denotes the mild solution to equation (\ref{eq:P}) with control \( u \) given by (\ref{eq:11}), then
\begin{align*}
x(t_2, u) &= \Psi_\alpha(t_2, \zeta)\left[x_0 - g(x)\right] + \int_{\zeta}^{t_2}\Psi_\alpha(t_2, s)\left[Bu(s) + F(s, x(b(s)))\right]\frac{ds}{s^{1-\alpha}} \\
&= \Psi_\alpha(t_2, \zeta)\left[x_0 - g(x)\right] - \int_{\zeta}^{t_2}\Psi_\alpha(t_2, s)B(L_\xi)^{-1} \\
&\quad \times \left[\Psi_\alpha(t_2, \zeta)\left[x_0 - g(x)\right] + \int_{\zeta}^{t_2}\Psi_\alpha(t_2, s)F(s, x(b(s)))\right](s)\frac{ds}{s^{1-\alpha}} \\
&\quad + \int_{\zeta}^{t_2} \Psi_\alpha(t_2, s)F(s, x(b(s)))\frac{ds}{s^{1-\alpha}}\\
& = 0.
\end{align*}

\end{proof}

Next, we introduce the operator \( Q \) defined on \( \mathbb{C}([\zeta,t_2]; Z) \) as follows:
\begin{align*}
    Qx(t) &= \Psi_\alpha(t,\zeta)[x_0 - g(x)]\\
    &\quad + \int_\zeta^t \Psi_\alpha(t,s)[-BH_\alpha(x_0 - g(x), F(s)) + F(s, x(b(s)))]\frac{ds}{s^{1-\alpha}}, \quad t \in (\zeta, t_2], 
\end{align*}

where \( u(t) \) is specified by equation (\ref{eq:11}). It suffices to demonstrate that \( Q \) possesses a fixed point in \( \mathbb{C}([\zeta,t_2]; Z) \), as such a fixed point corresponds to a mild solution of (\ref{eq:P}). To accomplish this, we begin by establishing two supporting lemmas.

For any fixed $n \in \mathbb{N}, \textit{and} \quad t \in [\zeta,t_2]$, we consider the following nonlocal problem:
\begin{equation}\label{eq:9}
    \begin{cases}
T_\alpha x(t) + Ax(t)= - BH_\alpha \left( x_0 - \Psi_\alpha\left( \frac{(n+1)\zeta}{n}, \zeta \right) g(x), F)(t)\right) + F(t, x(b(t))) ,  \\
x(\zeta) + \Psi_\alpha\left( \frac{(n+1)\zeta}{n},\zeta \right)g(x) = x_0.
\end{cases}
\end{equation}

\begin{lemma}\label{Lemma:3}
Provided that the conditions outlined in Theorem \ref{Theorem:1} hold true, the nonlocal problem~\eqref{eq:9} admits a mild solution on the interval $[\zeta, t_2]$ for every $n \in \mathbb{N}$.
\end{lemma}

\begin{proof}
\textbf{Step 1:}
 The control $u(\cdot)=-H_\alpha(x_0,F)(s)$ is bounded on $W_r.$ Indeed ,
\begin{align*}
    \lVert u \rVert & = \bigg( \int_{\zeta}^{t_2} \lVert H_\alpha \left( x_0 - \Psi_\alpha\left( \frac{(n+1)\zeta}{n}, \zeta \right) g(x), F \right)(\tau) \rVert^2 \, d\tau \bigg)^\frac{1}{2} \\
    &\leq \lVert H_\alpha \rVert\bigg( \lVert  x_0 - \Psi_\alpha\left( \frac{(n+1)\zeta}{n}, \zeta \right) g(x)\rVert +\bigg(\int_{\zeta}^{t_2} \bigg(   \lVert F(\tau,x(b(\tau))\rVert\, \bigg)^2 \, d\tau \bigg)^\frac{1}{2}\bigg)\\
    &\leq \lVert H_\alpha \rVert\bigg( \lVert x_0\rVert+M \varphi_r +\bigg(\int_{\zeta}^{t_2} \bigg(   \lVert h_r(\tau)\rVert\, \bigg)^2 \,  d\tau \bigg)^\frac{1}{2}\bigg)\\
    & \leq  \lVert H_\alpha \rVert\bigg( \lVert x_0\rVert+M \varphi_r +\bigg(\int_{\zeta}^{t_2} \bigg( \psi_r \bigg)^2 \,  d\tau \bigg)^\frac{1}{2}\bigg)\\
    & =  \lVert H_\alpha \rVert\bigg( \lVert x_0\rVert+M \varphi_r +\psi_r \sqrt{t_2-\zeta}\bigg).
\end{align*}

\textbf{Step 2:}
We aim to demonstrate that the operator \( Q_n \) possesses a fixed point in the Banach space \( \mathbb{C}([\zeta, t_2]; Z) \). To this end, we first verify the existence of a radius \( r > 0 \) such that \( Q_n \) maps the closed ball \( W_r \subset \mathbb{C}([\zeta, t_2]; Z) \) into itself. 

Assume, for contradiction, that no such \( r \) exists. Then, for each \( r > 0 \) and any \( x(\cdot) \in W_r \), there exists a time \( t(r) \in [\zeta, t_2] \) satisfying
\[
\| (Q_n x)(t(r)) \| > r.
\]
However, by invoking estimate (\ref{eq:3}) along with the assumptions \( (H_1)-(H_3) \), we derive a contradiction to this assertion.

\begin{align*}
   r &< \lVert (Q_n x)(t) \rVert \\
&\leq \biggl\lVert \Psi_\alpha(t,\zeta) \left[ x_0 - \Psi_\alpha\left( \frac{(n+1)\zeta}{n}, \zeta \right) g(x) \right]\biggr \rVert \\
&\quad + \left\lVert \int_\zeta^t \Psi_\alpha(t,s) B H_\alpha \left( x_0 - \Psi_\alpha\left( \frac{(n+1)\zeta}{n}, \zeta \right) g(x), F \right)(s) \frac{ds}{s^{1-\alpha}} \right\rVert \\
&\quad + \left\lVert \int_\zeta^t \Psi_\alpha(t,s) F(s, x(b(s))) \frac{ds}{s^{1-\alpha}} \right\rVert \\
\end{align*}
\begin{align*}
&\leq M\left( \lVert x_0 \rVert + M \lVert g(x) \rVert \right)\\
&\quad+ M \lVert B \rVert  \times \int_\zeta^{t_2} \left\lVert H_\alpha \left( x_0 - \Psi_\alpha\left( \frac{(n+1)\zeta}{n}, \zeta \right) g(x), F \right)(s) \right\rVert \frac{ds}{s^{1-\alpha}} \\
&\quad + M \int_\zeta^t \lVert F(s, x(b(s))) \rVert \frac{ds}{s^{1-\alpha}} \\
&\leq M\left( \lVert x_0 \rVert + M \varphi_r \right) + M \lVert B \rVert \left[{\int_{\zeta}^{t_2}  s^{\frac{\alpha -1}{1- \xi}}   }ds\right]^{1- \xi}\\
&\quad \quad \times \left( \int_\zeta^{t_2} \left\lVert H_\alpha \left( x_0 - \Psi_\alpha\left( \frac{(n+1)\zeta}{n}, \zeta \right) g(x), F \right)(s) \right\rVert^{\frac{1}{\xi}} ds \right)^{\xi}\\
&\quad  + M \left[{\int_{\zeta}^{t_2}  s^{\frac{\alpha -1}{1-\xi}}   }ds\right]^{1-\xi} \lVert h_r \rVert_{L^\frac{1}{\xi}} \\
&\leq M \left( \lVert x_0 \rVert + M \varphi_r \right)\\
&\quad+ M\lVert B \rVert \bigg(  (t_2)^{\frac{\alpha-\xi}{1-\xi}}-(\zeta)^{\frac{\alpha-\xi}{1-\xi}} \bigg)^{1-\xi} \bigg(\frac{1-\xi}{\alpha-\xi}\bigg)^{1-\xi} \lVert\\
&\quad \quad \times H_\alpha \rVert \left( \lVert x_0 + \Psi_\alpha\left( \frac{(n+1)\zeta}{n}, \zeta \right) g(x) \rVert + \lVert F \rVert_{L^\frac{1}{\xi}}\right)\\
&\quad+ M\bigg(  (t_2)^{\frac{\alpha-\xi}{1-\xi}}-(\zeta)^{\frac{\alpha-\xi}{1-\xi}} \bigg)^{1-\xi} \bigg(\frac{1-\xi}{\alpha-\xi}\bigg)^{1-\xi} \lVert h_r \rVert_{L^\frac{1}{\xi}}  \\
&\leq M \Bigg[ \lVert x_0 \rVert + M \varphi_r + \lVert B \rVert \lVert H_\alpha \rVert N \Bigg( \lVert x_0 \rVert + M \varphi_r + \lVert h_r \rVert_{L^\frac{1}{\xi}} \Bigg) + N \lVert h_r \rVert_{L^\frac{1}{\xi}} \Bigg]
\end{align*}
that is,
\begin{equation}\label{eq:eq_r}
    r \leq M \Bigg( \lVert x_0 \rVert + M \varphi_r + \lVert B \rVert \lVert H_\alpha \rVert N \Bigg( \lVert x_0 \rVert + M \varphi_r + \lVert h_r \rVert_{L^\frac{1}{\xi}} \Bigg) + N \lVert h_r \rVert_{L^\frac{1}{\xi}} \Bigg)
\end{equation}
Dividing both sides of equation~\eqref{eq:eq_r} by \( r \) and letting \( r \to +\infty \), we arrive at:

\[
M^2 L + \lVert B \rVert \lVert H_\alpha \rVert N M (ML + \gamma) + N M \gamma \geq 1,
\]

which contradicts equation~\eqref{eq:8}. Hence, there exists a constant \( r > 0 \) such that \( Q_n(W_r) \subseteq W_r \).

Next, we establish that \( Q_n \) is a completely continuous operator through the following steps.

\textbf{Step 3:} \( Q_n \) is continuous. Assume that \( x_k \to x \) in \( W_r \). Then for each \( t \), we have:

\begin{align*}
&\| (Q_n x_k)(t) - (Q_n x)(t) \|\\
& \leq \left\| \Psi_\alpha(t,\zeta) \Psi_\alpha\left( \frac{(n+1)\zeta}{n}, \zeta \right) [g(x_k) - g(x)] \right\| \\
&\quad + \left\| \int_\zeta^t \Psi_\alpha(t,s) \left[ 
B H_\alpha \left( x_0 - \Psi_\alpha\left( \frac{(n+1)\zeta}{n}, \zeta \right) g(x_k), F(s, x_k(b(s))) \right)(s) \right.\right. \\
&\qquad \left.\left. - B H_\alpha \left( x_0 - \Psi_\alpha\left( \frac{(n+1)\zeta}{n}, \zeta \right) g(x), F(s, x(b(s))) \right)(s) 
\right] \frac{ds}{s^{1-\alpha}} \right\| \\
&\quad + \left\| \int_\zeta^t \Psi_\alpha(t,s) [F(s, x_k(b(s))) - F(s, x(b(s)))] \, \frac{ds}{s^{1-\alpha}} \right\| \\
& \to 0 \quad \text{as } k \to +\infty.
\end{align*}

Since the functions \( g \), \( H_\alpha \), and \( F \) are continuous, and by applying the Lebesgue dominated convergence theorem, it follows that
\[
\| (Q_n x_k)(t) - (Q_n x)(t) \| \to 0, \quad \text{as } k \to +\infty,
\]
which confirms the continuity of \( Q_n \).

\textbf{Step 4:} We demonstrate that the collection of functions $\{(Q_n x)(\cdot) : x \in W_r \}$ forms an equicontinuous set on the interval $[\zeta,t_2]$ within the space $\mathbb{C}([\zeta,t_2]; Z)$.

Let $\zeta \leq t' < t'' \leq t_2$, with $x \in W_r$ and $\epsilon > 0$ being arbitrarily small, then

\begin{align*}
&\| Q_n x(t'') - Q_n x(t') \|
\leq  \Bigg\lVert \left[ \Psi_\alpha(t'', \zeta) - \Psi_\alpha(t', \zeta) \right] \left[ x_0 - \Psi_\alpha\left( \frac{(n+1)\zeta}{n}, \zeta \right) g(x) \right] \Bigg\rVert \\
&\quad + \Bigg\lVert \int_{\zeta}^{t'} \left[ \Psi_\alpha(t'', s) - \Psi_\alpha(t', s) \right]\\
&\quad \quad \times \left[ -BH_\alpha \left( x_0 - \Psi_\alpha\left( \frac{(n+1)\zeta}{n}, \zeta \right) g(x), F \right)(s) + F(s, x(b(s))) \right] \frac{ds}{s^{1-\alpha}} \Bigg\rVert \\
& \quad+ \Bigg\lVert \int_{t'}^{t''} \Psi_\alpha(t'', s) \\
&\quad \times\left[  - BH_\alpha \left( x_0 - \Psi_\alpha\left( \frac{(n+1)\zeta}{n}, \zeta \right) g(x), F \right)(s) + F(s, x(b(s))) \right] \frac{ds}{s^{1-\alpha}} \Bigg\rVert \\
\end{align*}
\begin{align*}
&\leq\Bigg\lVert \left[ \Psi_\alpha(t'', \zeta) - \Psi_\alpha(t', \zeta) \right] \left[ x_0 - \Psi_\alpha\left( \frac{(n+1)\zeta}{n}, \zeta \right) g(x) \right] \Bigg\rVert \\
& \quad+ \int_{\zeta}^{t' -\epsilon} \Bigg\lVert \Psi_\alpha(t'', s) - \Psi_\alpha(t', s) \Bigg\rVert\\
 &\quad \quad \times \Bigg\lVert -BH_\alpha\left( x_0 - \Psi_\alpha\left( \frac{(n+1)\zeta}{n}, \zeta \right) g(x), F \right)(s) + F(s, x(b(s))) \Bigg\rVert \frac{ds}{s^{1-\alpha}} \\
& \quad+ \int_{t'-\epsilon}^{t'} \Bigg\lVert \Psi_\alpha(t'', s) - \Psi_\alpha(t', s) \Bigg\rVert\\ 
&\quad \quad \times \Bigg\lVert - BH_\alpha \left( x_0 - \Psi_\alpha\left( \frac{(n+1)\zeta}{n}, \zeta \right) g(x), F \right)(s) + F(s, x(b(s))) \Bigg\rVert \frac{ds}{s^{1-\alpha}} \\
& \quad+ \int_{t'}^{t''} \Bigg\lVert \Psi_\alpha(t'', s)\\
&\quad \times\left[- BH_\alpha \left( x_0 - \Psi_\alpha\left( \frac{(n+1)\zeta}{n}, \zeta \right) g(x), F \right)(s) + F(s, x(b(s))) \right]\Bigg\rVert \frac{ds}{s^{1-\alpha}} \\
= & I_1 + I_2 + I_3 + I_4.
\end{align*}

Clearly, as \( t'' - t' \to 0 \), we observe that \( I_1 \to 0 \) because \( \Psi_\alpha(t, \zeta) \) is strongly continuous for \( t \geq \zeta \), and the set
\[
\left\{ x_0 - \Psi_\alpha\left( \frac{(n+1)\zeta}{n}, \zeta \right) g(x) : x \in W_r \right\}
\]
is compact within \( Z \). Under assumption (\ref{eq:3}), we can derive:

\begin{align*}
  I_3 &\leq 2M  \left(  (t')^{\frac{\alpha-\xi}{1-\xi}}-(t'- \epsilon)^{\frac{\alpha-\xi}{1-\xi}} \right)^{1-\xi} \left(\frac{1-\xi}{\alpha-\xi}\right)^{1-\xi}  \| B \| \\
    & \quad \times \| H_\alpha \| \left( \| x_0 \| + M \varphi_r + \| h_r \|_{L^{\frac{1}{\xi}}} \right) \\
    &+ 2 M \left(  (t_3)^{\frac{\alpha-\xi}{1-\xi}}-(t_3-\epsilon)^{\frac{\alpha-\xi}{1-\xi}} \right)^{1-\xi} \left(\frac{1-\xi}{\alpha-\xi}\right)^{1-\xi} \left( \int_{t_3-\epsilon}^{t_3} \lVert h_r(s)\rVert^{\frac{1}{\xi}} \, ds \right)^\xi.
\end{align*}

Similarly, for \( I_4 \), we have:

\begin{align*}
  I_4 &\leq M \cdot \left( (t'')^{\frac{\alpha-\xi}{1-\xi}} - (t')^{\frac{\alpha-\xi}{1-\xi}} \right)^{1-\xi} \\
    &\quad \times \left(\frac{1-\xi}{\alpha-\xi}\right)^{1-\xi} \| B \| \| H_\alpha \| \left( \| x_0 \| + M \varphi_r + \| h_r \|_{L^{\frac{1}{\xi}}} \right) \\
    &+ M \left(  (t_4)^{\frac{\alpha-\xi}{1-\xi}}-(t_3)^{\frac{\alpha-\xi}{1-\xi}} \right)^{1-\xi} \left(\frac{1-\xi}{\alpha-\xi}\right)^{1-\xi} \left( \int_{t_3}^{t_4} \lVert h_r(s)\rVert^{\frac{1}{\xi}} \, ds \right)^\xi.
\end{align*}

These estimates show that both \( I_3 \) and \( I_4 \) approach zero as \( t'' - t' \to 0 \), independently of \( x \). Furthermore, Lemma (\ref{lemma:1}) and similar evaluations guarantee that \( I_2 \to 0 \) as \( t'' - t' \to 0 \), ensuring that the set \( \{Q_n(x(t)): x \in W_r \} \) is equicontinuous on the interval \( [\zeta, t_2] \).

\textbf{Step 5:} We confirm that for any \( t \in [\zeta, t_2] \), the set \( \{ Q_n(x(t)), x \in W_r \} \) is relatively compact in \( Z \).

At \( t = \zeta \), we have \( Q_n(x(\zeta)) = x_0 - \Psi_\alpha\left( \frac{(n+1)\zeta}{n}, \zeta \right) g(x) \), and since \( g(x) \) is bounded in \( Z \), the result holds at \( t = \zeta \).

Next, define:
\begin{align*}
    x^\epsilon(t) &= \Psi_\alpha(t, \zeta) \left[ x_0 - \Psi_\alpha\left( \frac{(n+1)\zeta}{n}, \zeta \right) g(x) \right]\\
    &+ \int_\zeta^{t-\epsilon} \Psi_\alpha(t, s) \bigg[ - BH_\alpha \left( x_0 - \Psi_\alpha\left( \frac{(n+1)\zeta}{n}, \zeta \right) g(x), F \right)(s)\\
    &\quad+ F(s, x(b(s))) \bigg] \, \frac{ds}{s^{1-\alpha}}, \quad t \in (\zeta, t_2].
\end{align*}

Since \( \Psi_\alpha(t, s) \) is compact for \( t > s \geq \zeta \), the set \( Q_n^\epsilon(x(t)) = \{ x^\epsilon(t): x \in W_r \} \) is also relatively compact in \( Z \) for all \( 0 < \epsilon < t \). Furthermore, for each \( x \in W_r \),
\begin{align*}
   & \| Q_n(x)(t) - x^\epsilon(t) \|\\
   &\leq M \cdot \left( (t)^{\frac{\alpha-\xi}{1-\xi}} - (t-\epsilon)^{\frac{\alpha-\xi}{1-\xi}} \right)^{1-\xi} \left( \frac{1-\xi}{\alpha-\xi} \right)^{1-\xi} \| B \| \| H_\alpha \| \left( \| x_0 \| + M\varphi_r + \| h_r \|_{L^{\frac{1}{\xi}}} \right)\\
   &\quad + M \| h_r \|_{L^{\frac{1}{\xi}}} \left( (t)^{\frac{\alpha-\xi}{1-\xi}} - (t-\epsilon)^{\frac{\alpha-\xi}{1-\xi}} \right)^{1-\xi} \left( \frac{1-\xi}{\alpha-\xi} \right)^{1-\xi}.
\end{align*}

Taking the limit as \( \epsilon \to 0^+ \), we observe that sets arbitrarily close to \( \{ Q_n(x(t)): x \in W_r \} \) remain relatively compact. Thus, we conclude that \( Q_n(x(t)) \) is relatively compact in \( Z \).

By the infinite-dimensional version of the Ascoli-Arzela theorem, \( Q_n \) is a completely continuous operator on \( \mathbb{C}([\zeta, t_2]; Z) \).

Therefore, by Schauder's fixed-point theorem, the operator \( (Q_n x)(t) \) has at least one fixed point \( x_n(t) \), which is a mild solution to equation (\ref{eq:P}).

\end{proof}
Now, define the solution set $D$, $D(t)$, respectively, by
\[
D = \{x_n \in \mathbb{C}([\zeta,t_2]; Z) : x_n = Q_n x_n, \quad  n \geq 1\},
\]
\[
D(t) = \{x_n(t) : x_n \in D, n \geq 1\},\quad  t \in [\zeta,t_2].
\]

Then, one has
\begin{lemma}\label{Lemma:4}
    Assume that all the conditions of Theorem \ref{Theorem:1} are satisfied. Then for each $t \in (\zeta,t_2]$, $D(t)$ is relatively compact in $Z$ and $D$ is equicontinuous on $(\zeta,t_2]$.
\end{lemma}
\begin{proof}
    The proof is similar to that of Lemma \ref{Lemma:3}, and we omit it.
\end{proof}

\begin{proof}[\textbf{Proof of Theorem \ref{Theorem:1}}]\textbf{Part 2 :}

 We establish that the solution set $D$ of Eq. (\ref{eq:9}) is relatively compact in the space $ \mathbb{C}([\zeta, t_2]; Z)$. To demonstrate this, we focus on two key conditions: verifying that $D(\zeta)$ is relatively compact in $Z$ and confirming that $D$ exhibits equicontinuity at $t = \zeta$, as outlined in Lemma \ref{Lemma:4}.

Let $x_n \in D$ for all $n \geq 1$, and define the sequence $\bar{x}_n(t)$ as follows:
\[
\bar{x}_n(t) =
\begin{cases}
x_n(\delta), & \text{if } t \in [\zeta, \delta], \\
x_n(t), & \text{if } t \in [\delta, t_2],
\end{cases}
\]
where $\delta$ is chosen according to condition $(H_2)$. Under this condition, we have $g(x_n) = g(\bar{x}_n)$.

Furthermore, invoking Lemma \ref{Lemma:4}, it follows that the sequence $\{\bar{x}_n(\cdot) : n \in \mathbb{N}\}$ is relatively compact in $ \mathbb{C}([\zeta, t_2]; Z)$. Without loss of generality, assume that $\bar{x}_n(\cdot)$ converges to some $x^1(\cdot)$ in $ \mathbb{C}([\zeta, t_2]; Z)$ as $n \to \infty$. Utilizing the continuity of $\Psi_\alpha(t,s)$ and $g$, we then obtain

\begin{align*}
    &\biggl \lVert x_n(\zeta) - [x_0 - g(x^1)] \biggr \rVert = \biggl \lVert x_0 - \Psi_\alpha\left(\frac{(n+1)\zeta}{n}, \zeta\right) g(x_n) - x_0 + g(x^1)\biggr \rVert\\
    &\leq \biggl \lVert \Psi_\alpha\left(\frac{(n+1)\zeta}{n}, \zeta\right) g(x_n) - \Psi_\alpha\left(\frac{(n+1)\zeta}{n}, \zeta\right) g(x^1)\biggr \rVert\\
    &\quad + \biggl \lVert\Psi_\alpha\left(\frac{(n+1)\zeta}{n}, \zeta\right) g(x^1) - g(x^1)\biggr \rVert\\
    &\leq M \biggl \lVert g(\bar{x}_n(t)) - g(x^1)\| + \biggl \lVert\Psi_\alpha\left(\frac{(n+1)\zeta}{n}, \zeta\right) g(x^1) - g(x^1)\biggr \rVert\\
    &\to 0 \text{ as } n \to \infty.
\end{align*}

So, $D(\zeta)$ is relatively compact in $Z$.

On the other hand, as $D(\zeta) = \{x_0 - \Psi_\alpha\left(\frac{(n+1)\zeta}{n}, \zeta\right) g(x_n) : x_n \in D\}^\infty_{n=1}$ is relatively compact in $Z$, for $t \in (\zeta,t_2]$, we have
\begin{align*}
    &\|x_n(t) - x_n(\zeta)\|\\
    & \leq \Bigg\lVert \Psi_\alpha(t,\zeta)x_0 - x_0 + \Psi_\alpha\left(t, \zeta \right)\Psi_\alpha\left(\frac{(n+1)\zeta}{n}, \zeta\right)g(x_n) - \Psi_\alpha\left(\frac{(n+1)\zeta}{n}, \zeta\right)g(x_n)\Bigg\rVert\\
    &+ \Bigg\lVert \int_\zeta^t \Psi_\alpha(t,s) \left[ BH \left( x_0 - \Psi_\alpha\left( \frac{(n+1)\zeta}{n}, \zeta\right)g(x_n), F\right)(s) + F(s, x_n(b(s))) \right] \frac{ds}{s^{1-\alpha}} \Bigg\rVert\\
    &\leq \|\Psi_\alpha(t,\zeta)x_0 - x_0\| + \|\Psi_\alpha(t,\zeta) - I\| \left\lVert \Psi_\alpha\left(\frac{(n+1)\zeta}{n}, \zeta\right)g(x_n) \right\rVert\\
    &\quad+ \int_\zeta^t \Psi_\alpha(t,s) \left[ BH_\alpha\left( x_0 - \Psi_\alpha\left( \frac{(n+1)\zeta}{n}, \zeta \right) g(x_n), F\right)(s)  + F(s, x_n(b(s))) \right] \frac{ds}{s^{1-\alpha}}\\
    &\to 0.
\end{align*}

uniformly for all \( n \in \mathbb{N} \) as \( t \to \zeta \). This confirms that the set \( D \subseteq \mathbb{C}([\zeta, t_2]; Z) \) is equicontinuous at \( t = \zeta \). Consequently, \( D \) is relatively compact in \( \mathbb{C}([\zeta, t_2]; Z) \). Therefore, we can assume that \( x_n \to x^* \in \mathbb{C}([\zeta, t_2]; Z) \) as \( n \to \infty \) for some limit function \( x^* \).

Because
\begin{align*}
x_n(t) &= \Psi_\alpha(t,\zeta) \left[ x_0 - \Psi_\alpha\left( \frac{(n+1)\zeta}{n}, \zeta \right) g(x_n) \right] \\
&\quad + \int_\zeta^t \Psi_\alpha(t,s) \bigg[ - BH_\alpha\left( x_0 - \Psi_\alpha\left( \frac{(n+1)\zeta}{n}, \zeta \right) g(x_n), F \right)(s) \\
&\quad + F(s, x_n(b(s))) \bigg] \frac{ds}{s^{1-\alpha}}.
\end{align*}

taking the limit as $n \to \infty$ on both sides, we obtain
\begin{align*}
    &x^*(t) = \Psi_\alpha(t,\zeta) \left[ x_0 - g(x^*) \right]\\
    &+ \int_\zeta^t \Psi_\alpha(t,s) \bigg[ - BH_\alpha\bigg( x_0 - g(x^*), F \bigg)(s) + F(s, x^*(b(s))) \bigg] \frac{ds}{s^{1-\alpha}}
\end{align*}

for $\zeta \leq t \leq t_2$, which proves that $x^*(\cdot)$ is a mild solution of (\ref{eq:P}). Consequently, system (\ref{eq:P}) is exactly null controllable on $[\zeta,t_2]$.

\end{proof}

\section{Application}
\label{sec:Application}
As an illustration of Theorem \ref{Theorem:1}, we investigate the following fractional-order partial differential control system:
\begin{equation}\label{eq:12}
    \begin{cases}
\frac{\partial^\alpha}{\partial t^\alpha}w(x,t) = \frac{\partial^2}{\partial x^2}w(x,t) + p(t)w(x,t) + u(x,t) + h(t, w(x, b(t))),\quad \alpha \in (0,1], \\
w(0,t) = w(\pi,t) = 0, \quad t \in [0,T],  \\
w(x,\zeta) + \sum_{i=1}^{q}c_iw(x, t'_i) = w_0, \quad x \in [0, \pi], \zeta \in [0,T] \cap [0,(\alpha T)^{\frac{1}{\alpha}}],
\end{cases}
\end{equation}

Here, the time delays $t'_i \ (i = 1,2,\dots,q) \in (\zeta,T)$, and the functions $p(t): [\zeta,T] \to \mathbb{R}^+$ and $b(\cdot): [\zeta,T] \to [\zeta,T]$ are continuous. Define $Z = L^2([0,\pi])$, and introduce the operator family $A(t)$ as follows (see \cite{curtain2012introduction}, page 131):
\begin{equation}
    A(t)z = z'' + p(t)z
\end{equation}
with a common domain:
\[
D_A = \{ z(\cdot) \in Z : z, z' \text{ are absolutely continuous}, z'' \in Z, \text{ and } z(0) = z(\pi) = 0 \}.
\]

Next, we verify that the operator $A(t)$ generates an evolution family $\Psi_\alpha(t, s)$ satisfying the assumptions (a)–(d). The evolution operator is given by:
\[
\Psi_\alpha(t, s) = S_s(t-s)\exp\left(-\int_{\frac{s^\alpha}{\alpha}}^{\frac{t^\alpha}{\alpha}} p(\tau)\,d(\tau,\alpha)\right),
\]
where $S_s(t)$ for $\frac{t^\alpha}{\alpha} \in [0,(\alpha T)^{\frac{1}{\alpha}}]$ is a compact $C_0$-semigroup generated by the operator $A$, with
\[
Az = z''.
\]

For any $z \in D_A$, the following expansion is valid. It is easy to verify that the operator $A$ has a discrete spectrum with eigenvalues $-n^2$ for $n \in \mathbb{N}$, and the corresponding normalized eigenfunctions are given by $e_n(x) = \sqrt{2} \sin(n x)$. Thus, for $z \in D_A$ (see \cite{curtain2012introduction}, p.~29), we can write:
\begin{equation}\label{eq:A(t)}
    A(t)z = \sum_{n=1}^\infty (-n^2 - p(t))\langle z, e_n \rangle e_n.
\end{equation}

The common domain of definition for all involved operators coincides with that of $A$. Furthermore, for each $z \in Z$, the representation of $\Psi_\alpha(t, s)z$ is given as follows.

\[ \Psi_\alpha(t, s)z = \sum_{n=1}^\infty e^{-n^2(\frac{t^\alpha}{\alpha}-\frac{s^\alpha}{\alpha})-\int_\frac{s^\alpha}{\alpha}^\frac{t^\alpha}{\alpha} p(\tau)\,d(\tau,\alpha)} \langle z, e_n \rangle e_n. \]

Therefore, for every $z \in Z,$
\begin{align*}
    & \Psi_\alpha(t, s)z =\Psi_\alpha^*(t, s)z \\
    &= \sum_{n=1}^\infty e^{-n^2(\frac{t^\alpha}{\alpha}-\frac{s^\alpha}{\alpha})-\int_\frac{s^\alpha}{\alpha}^\frac{t^\alpha}{\alpha} p(\tau)\,d(\tau,\alpha)} \sin(nx) \int_{0}^{\pi} z(\beta) \sin(n\beta) \, d\beta, \quad z \in Z,
\end{align*}

To reformulate system~(\ref{eq:12}) in the structure of~(\ref{eq:P}), we define the mappings $F : [\zeta,T] \times Z \to Z$ and $g : C([\zeta,T], Z) \to Z$ by
\[
F(t,w) = h(t, w(x,t)), \quad g(w(x,t)) = \sum_{i=1}^{q} c_i w(x, t'_i).
\]
It is clear that $g(\cdot)$ fulfills the condition $(H_2)$ since we can choose $\delta = t'_1$. Regarding the function $h : [\zeta,T] \times \mathbb{R} \to \mathbb{R}$, we assume the following:
\begin{itemize}
    \item[(i)] For each $x \in \mathbb{R}$, the mapping $h(\cdot, x)$ is measurable on $[\zeta,T]$;
    \item[(ii)] For every fixed $t \in [\zeta,T]$, the function $h(t, \cdot)$ is continuous, and there exists a constant $c \geq 0$ such that
    \[
    |h(t,x)| \leq c|x|, \quad \text{for all } t \in [\zeta,T].
    \]
\end{itemize}

Under these assumptions, $F(t,w)$ satisfies condition $(H_1)$. Let us now take $u \in L^2([0,1], X)$, and set $B = I$, which gives $B^* = I$. Therefore, system~(\ref{eq:12}) can be rewritten in the form of~(\ref{eq:P}). Next, we consider the linear system with an additional term $\phi \in L^{\frac{1}{\xi}}([0,\pi], Z)$:

\begin{equation}
\begin{cases}
\frac{\partial^{\alpha} z(x, t)}{\partial t^{\alpha}} = \frac{\partial^2 z}{\partial x^2}(x, t) + p(t)z(x, t) + \mu(x, t) + \phi(t), \quad \alpha \in (0,1]  \\
z(0, t) = z(\pi, t) = 0, \quad t\in [0,T], \\
z(x, \zeta) = z_0, \quad x \in [0, \pi],\quad \zeta\in [0,T]\cap [0,(\alpha T)^{\frac{1}{\alpha}}],
\end{cases}
\label{eq:13}
\end{equation}

According to Remark \ref{remark:1}, the exact null controllability of Equation (\ref{eq:13}) is equivalent to the existence of a constant $\gamma > 0$ satisfying the inequality
\begin{align*}
    \int_{\zeta}^{T} \|B^*\Psi_\alpha^*(T,s)z\|^2 d(s,\alpha) \geq \gamma \left(\| \Psi_\alpha^*(T,\zeta)z\|^2 + \int_{\zeta}^{T} \|\Psi_\alpha^*(T,s)z\|^2 d(s,\alpha) \right),
\end{align*}
which is also equivalent to
\begin{equation}
    \int_{\zeta}^{T} \|\Psi_\alpha(T,s)z\|^2 d(s,\alpha) \geq \gamma \left(\| \Psi_\alpha(T,\zeta)z\|^2 + \int_{\zeta}^{T} \|\Psi_\alpha(T,s)z\|^2 d(s,\alpha) \right).
    \label{eq:14}
\end{equation}

In the work of Curtain and Zwart \cite{curtain2012introduction}, it is demonstrated that the linear system given by (\ref{eq:13}), with $\phi = 0$, is exactly null controllable provided that
\[
\int_{\zeta}^{T} \|\Psi_\alpha(T,s)z\|^2 d(s,\alpha) \geq (T - \zeta) \| \Psi_\alpha(T,\zeta)z\|^2.
\]
From this, one can derive the inequality
\[
\frac{1}{1+(T-\zeta)} \int_{\zeta}^{T} \|\Psi_\alpha(T,s)z\|^2 d(s,\alpha) \geq \frac{T-\zeta}{1+(T-\zeta)} \| \Psi_\alpha(T,\zeta)z\|^2,
\]
which leads to
\[
\int_{\zeta}^{T} \|\Psi_\alpha(T,s)z\|^2 d(s,\alpha) \geq \frac{T-\zeta}{1+(T-\zeta)} \left(\| \Psi_\alpha(T,\zeta)z\|^2 + \int_{\zeta}^{T} \|\Psi_\alpha(T,s)z\|^2 d(s,\alpha) \right).
\]

Thus, setting $\gamma = \frac{T - \zeta}{(T - \zeta) + 1}$ ensures that inequality (\ref{eq:14}) holds. As a result, the system represented by Equation (\ref{eq:13}) is exactly null controllable on the interval $[\zeta, T]$. By Theorem \ref{Theorem:1}, this implies that the system described by Equation (\ref{eq:12}) also achieves exact null controllability over the same interval, provided the condition stated in inequality (\ref{eq:8}) is fulfilled.

\begin{remark}
  A semilinear function that meets the conditions $H_1$ and $H_3$ is given by
  \[
  h(t, w(x, b(t))) = \sin(x) + (b(t))^{2\alpha} \sin(2x),
  \]
  where $b(t) = \cos(t)$. Furthermore, by choosing $g(x) = \cos^2(x)$, which satisfies condition $H_2$, and $p(t) = t^3 + t^2 + 1$, the system (\ref{eq:12}) becomes exactly null controllable.
\end{remark}

% References
\bibliographystyle{unsrtnat} % Use numerical references
\bibliography{references} % Replace 'references' with your .bib file

\end{document}